\documentclass{article}
\usepackage[margin=1in]{geometry}
\usepackage{verbatim}

\usepackage{amsmath}
\usepackage{amssymb}
\usepackage{amsthm}
\usepackage{amscd}
\usepackage{enumerate}
\usepackage[colorlinks=true]{hyperref}
\usepackage{bbm}
\usepackage{etoolbox}

\usepackage[utf8]{inputenc}

\newcommand{\tomemail}{\href{mailto:tom.bachmann@zoho.com}{tom.bachmann@zoho.com}}

\ifdefvoid{\longdocument}{
\newtheorem*{defn}{Definition}
\newtheorem{prop}{Proposition}
}{
\newtheorem{prop}{Proposition}[chapter]
\newtheorem{defn}[prop]{Definition}
}
\newtheorem{corr}[prop]{Corollary}
\newtheorem{lemm}[prop]{Lemma}
\newtheorem{thm}[prop]{Theorem}

\newtheorem*{thm*}{Theorem}
\newtheorem*{corr*}{Corollary}
\newtheorem*{prop*}{Proposition}
\newtheorem*{lemm*}{Lemma}
\theoremstyle{definition}

\theoremstyle{remark}
\newtheorem{rmk}[prop]{Remark}

\newcommand{\Aone}{{\mathbb{A}^1}}
\newcommand{\Gm}{{\mathbb{G}_m}}

\newcommand{\Hom}{\operatorname{Hom}}

\DeclareRobustCommand{\ul}{\underline}

\newcommand{\et}{{\acute{e}t}}

\newcommand{\iso}{\cong}

\title{$\Aone$-invariants in Galois cohomology and a claim of Morel}
\author{Tom Bachmann \\ \tomemail}

\let\cat=\mathrm
\def\Fin{\cat F\mathrm{in}}
\let\bb=\mathbb
\def\A{\bb A}

\begin{document}

\maketitle
\begin{abstract}
We establish variants of the results in \cite{garibaldi2003cohomological} for
invariants taking values in a strictly homotopy invariant sheaf. As
an application, we prove the folklore result of Morel that $\ul{\pi}_0^{\Aone} (B_\et
\Fin^{bij})^+ = \ul{GW}$.
\end{abstract}

\section{$\Aone$-invariants of Étale Algebras}
This entire section is basically a minor variation of arguments from
\cite{garibaldi2003cohomological}. Related results were also obtained by Morel
(unpublished) and Hirsch \cite[Theorem 2.3.12]{hirsch2018cohomological}.

Throughout, we fix a field $k$. Let $Sm(k)$ denote the category of
smooth $k$-schemes and $Pre(Sm(k))$ the category of presheaves on $Sm(k)$. As
usual, if $F \in Pre(Sm(k))$ and $X$ is an essentially smooth $k$-scheme, then
$F(X)$ makes sense.

\begin{defn}
We call $F \in Pre(Sm(k))$ \emph{homotopy invariant} if $F(X) = F(X \times
\Aone)$ for all $X \in Sm(k)$.
\end{defn}
\begin{defn}
We call $F \in Pre(Sm(k))$ \emph{unramified} if $F(X \coprod Y) = F(X) \times
F(Y)$, for all connected $X \in Sm(k)$ the canonical map $F(X) \to F(k(X))$ is
injective, and moreover
\[ F(X) = \bigcap_{x \in X^{(1)}} F(X_x) \]
(intersection in $F(k(X))$).
\end{defn}

\begin{rmk}
Suppose that $k$ is perfect and $F \in Pre(Sm(k))$ is a strictly homotopy
invariant sheaf of abelian groups. Then $F$ is unramified \cite[Lemma 6.4.4]{morel2005stable}.
\end{rmk}

In this section we are interested in the presheaf $Et_n \in Pre(Sm(k))$ which
assigns to $X \in Sm(k)$ the set of isomorphism classes of étale $X$-schemes,
everywhere of rank $n$. (This is neither homotopy invariant nor unramified, of
course.)

\begin{defn}
A rank $n$ versal étale scheme is an étale morphism $X \to Y$, everywhere of
rank $n$, such that if $T \to Spec(l)$ is any rank $n$ étale algebra over a
finitely generated field extension $l/k$, then $T$ is obtained from $X \to Y$ by
pullback along a morphism $Spec(l) \to Y$.
\end{defn}

We will make good use of the following result.

\begin{thm}[\cite{garibaldi2003cohomological}, Proposition 24.6(2)]
\label{thm:versal-torsor}
There exists a smooth $k$-scheme $X$, an irreducible divisor $\Delta
\subset \A^n$ and a finite étale morphism $p: X \to \mathbb{A}^n \setminus
\Delta$ of rank $n$, such that $p$ is versal.

Moreover, let $\eta$ be the generic point of $\mathbb{A}^n$, $L_\Delta^h$ the
Henselization of $\mathbb{A}^n$ in the generic point of $\Delta$,
and $\eta_h$ the generic point of
$L_\Delta^h$. The finite étale $\eta_h$-scheme $X_{\eta_h}$ splits as a disjoint union
$X_{\eta_h} = X_1 \coprod X_2$ with $X_1$ of rank two (unless $n=1$).
\end{thm}
\begin{proof}
Let us review the construction of $p$. Let $p_0: X_0 = Spec(k[x_1, \dots, x_n][t]/(t^n
+ x_1 t^{n-1} + \dots + x_n)) \to \A^n$ be the evident map. Let
$\Delta \subset \A^n$ be the branch locus of $p_0$; the reference proves that this is an
irreducible divisor. Since étale algebras over fields are simple, it is clear
that $X := X_0 \setminus p_0^{-1}(\Delta)$ is versal.

Let $L_\Delta^\wedge$ be the completion of $L^h_\Delta$ and write $\hat{\eta}$ for
the generic point of $L_\Delta^\wedge$. The reference proves that $X_{\hat \eta}
= Y_1 \coprod Y_2$, with $Y_1$ of rank two and $Y_2$ \emph{unramified}. What
this means is that there exists a finite étale morphism $Y_2' \to
L_\Delta^\wedge$ with $Y_2 \iso Y_2' \times_{L_\Delta^\wedge} \hat{\eta}$
\cite[beginning of Section 11, Proposition 24.2(3) and Definition
24.3]{garibaldi2003cohomological}. Since $L_\Delta^\wedge \to L_\Delta^h$ is a
morphism of henselian local rings inducing
an isomorphism on residue fields, it follows that there exists a finite étale
morphism $Y_2'' \to L_\Delta^h$ with $Y_2' \iso Y_2'' \times_{L_\Delta^h}
L_\Delta^\wedge$ \cite[Tag 04GK]{stacks-project}. Lemma
\ref{lemm:hensel-completion}(2) below then furnishes us with a retraction
$(Y_2'')_{\eta_h} \to X_{\eta_h} \to (Y_2'')_{\eta_h}$. Thus $X_{\eta_h}$ splits as
$X_1 \coprod (Y_2'')_{\eta_h}$, and $X_1$ must have the same degree as $Y_1$, i.e.
2. This concludes the proof.
\end{proof}

We used the following result, which is surely well-known.
\begin{lemm} \label{lemm:hensel-completion}
Let $R$ be a henselian DVR with completion $\hat R$, fraction field $K$ and
completed fraction field $\hat K$.
\begin{enumerate}
\item Let $A/K$ be an étale algebra with completion $\hat{A}/\hat{K}$. If $x \in
  \hat{A}$ satisfies a separable polynomial with coefficients in $A$, then $x
  \in A$.
\item Let $A, B$ be étale $K$-algebras. Then $\Hom_K(A, B) = \Hom_{\hat K}(\hat
  A, \hat B)$.
\end{enumerate}
\end{lemm}
\begin{proof}
(1) We may assume that $A=L$ is a field. The normalization $R'$ of $R$ in $L$ is
finite \cite[Tag 032L]{stacks-project}, and hence $R'$, being a domain, is local henselian
\cite[Tag 04GH(1)]{stacks-project}. We may thus replace $R$ by $R'$ and assume
that $A = K$. Let $\pi$ be a uniformizer of $R$; then $\hat K = \hat R[1/\pi]$.
It follows that $\pi^n x \in R$ for $n$ sufficiently large and still satisfies a
separable polynomial; hence we may assume that $x \in R$. This reduced statement
is a well-known characterisation of henselian DVRs.\footnote{It is proved for
example here: \url{https://mathoverflow.net/q/105891}.}

(2) We may assume that $A=K[T]/P$, where $P$ is a separable polynomial. Then
$\Hom_{\hat K}(\hat A, \hat B)$ is the set of elements $t \in \hat B$ with $P(t)
= 0$. By (1), such $t$ lie in $B$. It follows that $\Hom_K(A, B) \to \Hom_{\hat
K}(\hat A, \hat B)$ is surjective. Injectivity is clear since $A \to \hat A$
etc. are all injective.
\end{proof}

\begin{lemm} \label{lemm:hensel}
Let $X$ be the localisation of a smooth scheme in a point of codimension one.
Write $X^h$ for the Henselization, $\eta \in X$ for the generic point and
$\eta^h$ for the generic point of $X^h$. If $F \in Pre(Sm(k))$ is a Nisnevich
sheaf, then the following diagram is cartesian:
\begin{equation*}
\begin{CD}
F(X) @>>> F(\eta) \\
@VVV      @VVV    \\
F(X^h) @>>> F(\eta^h).
\end{CD}
\end{equation*}
\end{lemm}
\begin{proof}
Let $X' \to X$ be an étale neighbourhood of the closed point, and $\eta'$ the
generic point of $X'$. Then
\begin{equation*}
\begin{CD}
\eta' @>>> X' \\
@VVV      @VVV \\
\eta  @>>> X
\end{CD}
\end{equation*}
is a distinguished Nisnevich square; hence applying $F$ yields a cartesian
square. Since $X^h$ is obtained as the filtered inverse limit of the $X'$ and
filtered colimits (of sets) commute with finite limits, the result follows.
\end{proof}

Recall that an étale algebra $A/k$ is called \emph{multiquadratic} if it is a
(finite) product of copies of $k$ and quadratic separable extensions of $k$.
\begin{corr} \label{corr:et-invariants}
Let $F \in Pre(Sm(k))$ be a homotopy invariant, unramified Nisnevich sheaf of
sets and $a: Et_n \to F$ any morphism (of presheaves of sets). Assume there
exists $* \in F(*)$ such that for any field $l/k$ and any multiquadratic étale
algebra $A/l$ we have $a(A) = *|_l$. Then for any $Y \in Sm(k)$ and any $A \in
Et_n(Y)$ we have $a(A) = *|_l$.
\end{corr}
\begin{proof}
This is essentially the same as the proof of \cite[Theorem
24.4]{garibaldi2003cohomological}.

Since $F$ is unramified, it suffices to prove the claim when $Y$ is the spectrum
of a field. We proceed by induction on $n$. If $n\in \{1,2\}$ there is nothing
to do.

Suppose now that $l/k$ is a field, $A/l \in Et_n(l)$ and $A \approx A_1
\times A_2$ with $A_2$ multiquadratic (but non-zero). Define an invariant $a'$
of $Et_{n-2}$ over $l$ via
$a'(B) = a(B \times A_2)$. By assumption $a'(B) = *$ if $B$ is multiquadratic,
hence $a' = *$ by induction. We conclude that $a(A) = *$.

Now let $X \to \mathbb{A}^n \setminus \Delta$ be the versal morphism
from Theorem \ref{thm:versal-torsor}. We
consider $a(X_\eta) \in F(K)$, where $K = k(\mathbb{A}^n)$. I claim that if $x
\in \mathbb{A}^n$ is a point of codimension one, then $a(X_\eta)$ is in the
image of $F(\mathbb{A}^n_x) \to F(K)$. If $x \not\in \Delta$ this is clear, because
then $x \in \mathbb{A}^n \setminus \Delta$ and $a(X_\eta) = a(X_x)|_\eta$.
We thus need to deal with the case
where $x$ is the generic point of $\Delta$.
By Theorem  \ref{thm:versal-torsor}, $X_{\eta_h}$ splits off a quadratic
factor. Thus $a(X_\eta)|_{\eta_h} = a(X_{\eta_h}) = *$, by the previous step.
The claim now follows from Lemma \ref{lemm:hensel}.

Since $F$ is unramified, it follows that
$a(X_\eta) \in F(K)$ lies in the image of $F(\mathbb{A}^n) \to
F(K)$. But $F(\mathbb{A}^n) = F(*)$, and so there exists $a_0 \in F(*)$ such
that $F(X_\eta) = a_0|_K$. Since $X_\eta$ is versal this implies that $a(A/l) =
a_0|_l$ for any $l$ and any $A$. But $a(k/k) = *$, so $a_0 = *$. This concludes
the induction step.
\end{proof}

\section{Application: computing $\ul{\pi}_0^{\Aone} (B_\et \Fin^{bij})^+$}
Now let $k$ be a perfect field of characteristic not two.

We write $Shv_{Nis,\Aone}^{ab}(k)$ for the category of strictly homotopy
invariant Nisnevich sheaves (of abelian groups) on $Sm(k)$, $Pre^{mon}(k)$ for
the category of presheaves of monoids, with morphisms the morphisms of monoids. We
have obvious forgetful functors
\[ Shv_{Nis,\Aone}^{ab}(k) \to Pre^{mon}(k) \to Pre(Sm(k)). \]
We write $U: Shv_{Nis,\Aone}^{ab}(k) \to Pre(k)$, and
$U_{mon}: Shv_{Nis,\Aone}^{ab}(k) \to Pre^{mon}(k)$. Then the functors
$U_{mon}$ and $U$ have (potentially partially defined) left adjoints denoted
$F_{mon}$ and $F$. We make use of the following result of Morel.

\begin{thm}[\cite{A1-alg-top}, Theorem 3.46] \label{thm:morel}
The morphism of presheaves of sets $\Gm/2 \to \ul{GW}$, $[a] \mapsto \langle a
\rangle$ exhibits $\ul{GW}$ as $F(\Gm/2)$.
\end{thm}

Let $Et_* \in Pre^{mon}(k)$ denote the presheaf of monoids $X \mapsto \coprod_{n
\ge 0} Et_n(X)$; the monoidal operation is given by disjoint union of étale
schemes. For an étale algebra $A/l$, denote by $tr(A) \in \ul{GW}(l)$ the class
of its trace form. 
We shall now prove the result advertised in the heading, in the following guise.

\begin{prop}
Let $k$ be a perfect field of characteristic not two.

The morphism of presheaves of monoids $tr: Et_* \to \ul{GW}$, $A \mapsto tr(A)$
exhibits $\ul{GW}$ as $F_{mon}(Et_*)$.
\end{prop}
\begin{proof}
Let $\phi: Et_* \to U_{mon}F$ be any morphism, where $F \in
Shv_{Nis,\Aone}^{ab}(k)$ is arbitrary. For $a \in \mathcal{O}(X)^\times$ let
$X_a := X[t]/(t^2 - a)$. Since we are in characteristic not two, $X_a \to X$ is
étale. Define $t: \Gm/2 \to U F$ by mapping $a \in \mathcal{O}(X)^\times$ to
\[ t([a]) := \phi([X_{a/2}]) - \phi([X_2]) + \phi([X]) \in F(X). \]
(The reason for this formula is
that $tr([X_a]) = \langle 2 \rangle + \langle 2a \rangle$, and hence $tr([X_{a/2}])
- tr([X_2]) + tr([X]) = \langle a \rangle$.) By Theorem
\ref{thm:morel} this induces $t': \ul{GW} \to F$. I claim that the following
diagram commutes:
\begin{equation*}
\begin{CD}
Et_* @>\phi>>  F \\
@VtrVV        @| \\
\ul{GW} @>t'>> F \\
\end{CD}
\end{equation*}
By Corollary \ref{corr:et-invariants} it suffices to show this for
multi-quadratic algebras over fields $l$. But we are dealing with morphisms of
monoids into unramified sheaves, so it suffices to show this for $l/l \in
Et_1(l)$ and $l_a/l \in Et_2(l)$, where $l/k$ is a finitely generated field
extension. Now \[ t'(tr([l])) =
t'(\langle 1 \rangle) = t([1]) = \phi([l_{1/2}]) - \phi([l_2]) + \phi([l]) = \phi([l]) \]
since $l_2$ and $l_{1/2}$ are isomorphic. Finally \[ t'(tr([l_a])) = t'(\langle 2
\rangle + \langle 2a \rangle) = \phi([l_1]) - \phi([l_2]) + \phi([l]) +
\phi([l_a]) - \phi([l_2]) + \phi([l]) = \phi([l_a]) + 2(\phi([l_1]) -
\phi([l_2])) \] (note that $[l_1] = 2[l]$). Hence to prove the claim we need to show that $2(\phi([l_1]) -
\phi([l_2])) = 0$. Consider $u: \Gm/2 \to Et_*, a \mapsto [X_a]$. Then, applying
Theorem \ref{thm:morel} again, we get a commutative diagram
\begin{equation*}
\begin{CD}
\Gm/2 @>u>> Et_* \\
@VVV        @V{\phi}VV \\
\ul{GW} @>u'>> F.
\end{CD}
\end{equation*}
Since $2 \langle 2 \rangle = 2 \in GW(k)$ (indeed $2x^2 + 2y^2 = (x+y)^2 +
(x-y)^2$) and $\phi([l_a]) = \phi(u(a)) = u'(\langle a \rangle)$, this proves
the claim.

Consequently we have proved that any morphism $\phi$ factors through $tr: Et_*
\to \ul{GW}$. Since the image of $tr$ generates $\ul{GW}$ (as an unramified
sheaf of abelian groups), this factorization is unique. This concludes the
proof.
\end{proof}

\bibliographystyle{plainc}
\bibliography{bibliography}

\end{document}